\documentclass{amsart}
\usepackage{amssymb,amsmath, amsfonts, amsrefs, tikz, epsfig, float}
\usepackage{amsthm}
\usepackage{mathrsfs}
\usepackage{enumitem, indentfirst}
\usepackage[fleqn,tbtags]{mathtools}
\usepackage{color}
 \newtheorem{theorem}{Theorem}[section]

 \newtheorem{corollary}[theorem]{Corollary}
 \newtheorem{lemma}[theorem]{Lemma} 
 \newtheorem{proposition}[theorem]{Proposition}

 \theoremstyle{definition}

 \theoremstyle{remark}

 \newtheorem*{remark}{Remark}
  \numberwithin{equation}{section}

\renewcommand{\theta}{\vartheta}
\DeclareMathOperator{\sform}{\mathfrak{s}}
\DeclareMathOperator{\tform}{\mathfrak{t}}
\DeclareMathOperator{\tfrom}{\mathfrak{t}}

\DeclareMathOperator{\wform}{\mathfrak{w}}

\DeclareMathOperator{\rform}{\mathfrak{r}}

\DeclareMathOperator{\ran}{ran}

\DeclarePairedDelimiterX\sipt[2]{\langle}{\rangle^{}_{\tform}}{#1,#2}
\DeclarePairedDelimiterX\sips[2]{\langle}{\rangle^{}_{S}}{#1\,\delimsize\vert\,#2}

\DeclarePairedDelimiterX\sipv[2]{(}{)_{v}}{#1\,\delimsize\vert\,#2}
\DeclarePairedDelimiterX\sipw[2]{(}{)_{w}}{#1\,\delimsize\vert\,#2}

\newcommand{\dupN}{\mathbb{N}}
\newcommand{\seq}[1]{(#1_{n})_{n\in\dupN}}

\newcommand{\dom}{\operatorname{dom}}

\newcommand{\D}{\mathcal{D}}
\newcommand{\DD}{\mathcal{D_{*}}}
\newcommand{\DDD}{\mathcal{D_{**}}}

\newcommand{\hil}{\mathcal{H}}

\newcommand{\hilt}{\hil^{}_{\tform}}

\DeclarePairedDelimiterX\abs[1]{\lvert}{\rvert}{#1}
\DeclarePairedDelimiterX\sip[2]{(}{)}{#1,#2}
\DeclarePairedDelimiterX\sipr[2]{\langle}{\rangle_{\rform}}{#1,#2}
\DeclarePairedDelimiterX\siptilde[2]{(}{)_{\!_{\widetilde{A}}}}{#1\,\delimsize\vert\,#2}
\DeclarePairedDelimiterX\sipf[2]{(}{)_{f}}{#1\,\delimsize\vert\,#2}
\DeclarePairedDelimiterX\sipg[2]{(}{)_{g}}{#1\,\delimsize\vert\,#2}
\DeclarePairedDelimiterX\siptw[2]{(}{)_{\tform+\wform}}{#1\,\delimsize\vert\,#2}
\DeclarePairedDelimiterX\set[2]{\{}{\}}{#1\,:\,#2}
\DeclarePairedDelimiterX\dual[2]{\langle}{\rangle}{#1,#2}
\DeclarePairedDelimiterX\sipa[2]{(}{)_{\!_A}}{#1\,\delimsize\vert\,#2}
\DeclarePairedDelimiterX\sipe[2]{\langle}{\rangle_{\!_{\mathcal E}}}{#1\,\delimsize\vert\,#2}
\DeclarePairedDelimiterX\sipQ[2]{(}{)_{\!_{*}}}{#1\,\delimsize\vert\,#2}
\DeclarePairedDelimiterX\sipab[2]{(}{)_{\!_{A+B}}}{#1\,\delimsize\vert\,#2}
\DeclarePairedDelimiterX\sipb[2]{(}{)_{\!_B}}{#1\,\delimsize\vert\,#2}

\newcommand{\limn}{\lim\limits_{n\rightarrow\infty}}

\renewcommand{\epsilon}{\varepsilon}
\allowdisplaybreaks

\begin{document}
\title{Basic representation theorems of forms}

\author[Z. Sebesty\'en]{Zolt\'an Sebesty\'en}
\author[Zs. Tarcsay]{Zsigmond Tarcsay}
 
\address{%
Zs. Tarcsay \\ Department of Mathematics\\ Corvinus University of Budapest\\ IX. F\H ov\'am t\'er 13-15.\\ Budapest
H-1093 \\ Hungary\\ and Department of Applied Analysis  and Computational Mathematics\\ E\"otv\"os Lor\'and University\\ P\'azm\'any P\'eter s\'et\'any 1/c.\\ Budapest H-1117\\ Hungary}
\email{zsigmond.tarcsay@uni-corvinus.hu}

\thanks{ 
Project no. TKP 2021-NVA-09 has been implemented with the support provided by the Ministry of Innovation and Technology of Hungary from the National Research, Development and Innovation Fund, financed under the TKP2021-NVA funding scheme.}
\address{%
Z. Sebesty\'en \\ Department of Applied Analysis  and Computational Mathematics\\ E\"otv\"os Lor\'and University\\ P\'azm\'any P\'eter s\'et\'any 1/c.\\ Budapest H-1117\\ Hungary }
\email{sebesty@cs.elte.hu}



\dedicatory{Dedicated to Professor Henk de Snoo  on the occasion of his 80th birthday}

\begin{abstract}
We study maximal representations of nonnegative sesquilinear forms in real or complex Hilbert spaces, that are not necessarily closed or even closable. We associate positive self-adjoint operators with such forms, in a sense similar to Kato's representation theorems. In particular, we give a brief proof of the Friedrichs extension of a densely defined positive operator. 
\end{abstract}

\maketitle

\section{Introduction}

One of the most fundamental results in the theory of unbounded positive operators is the Kato representation theorem for nonnegative closed sesquilinear forms.  This theorem establishes a deep connection between such forms and positive self-adjoint operators on a Hilbert space: every densely defined nonnegative closed form $\tform$ can be represented by a positive self-adjoint operator $T$, in the sense that 
\begin{equation*}
    \sip{Tg}{h}=\tform(g,h),
\end{equation*}
see \cite{kato}. The operator $T$ is typically defined only on a proper subdomain of the form $\tform$, meaning that the above representation does not hold for all vectors $g, h \in \dom \tform$. A more complete result (which supplements the first representation theorem) is given by Kato's second representation theorem, which provides a sharper relationship between the form $\tform$ and the operator $T$ that represents it. According to this theorem, the domains of $\tform$ and $T^{1/2}$ coincide, and on this domain, the following identity holds:
\begin{equation*}
\sip{T^{1/2}g}{T^{1/2}h} = \tform(g,h).    
\end{equation*}
Here, $T^{1/2}$ denotes the positive self-adjoint square root of the operator $T$, which is typically defined using the spectral theorem. Throughout this paper, we allow the Hilbert space $\hil$ to be either complex or real. In the case of a real Hilbert space, the spectral theorem is not available; however, the existence of the square root of a positive self-adjoint operator can still be established in the real setting as well (see e.g. \cite{SZTZS_squareroot}).

In his 1929 paper \cite{Neumann}, J. von Neumann proved that any densely defined, lower semi-bounded symmetric operator admits a lower semi-bounded self-adjoint extension. Neumann conjectured that such an extension could also preserve the lower bound of the original operator; however, he was only able to prove the existence of extensions with a slightly smaller lower bound, i.e., with a lower bound reduced by an arbitrary $\epsilon > 0$. Neumann's conjecture was later proved independently by Stone \cite{Stone} and Friedrichs \cite{Friedrichs}. The construction in this latter paper essentially uses the closure of the form associated with the operator in question, and the operator defined through this closed form is what we now call the Friedrichs extension (cf. \cites{Freudenthal,SebTarcsayOpuscula,SebStoch2007,SebTarcsayWasaa}).
So, by means of Kato's representation theorem, the Friedrichs extension of the positive operator $T$ is, in fact, the positive self-adjoint operator that represents the closure of the form associated with $T$.

In this paper, we study possible representations of nonnegative sesquilinear forms that are not necessarily closed or even closable. Our aim is to describe how positive self-adjoint operators can still be meaningfully associated with forms beyond the scope of Kato's original theorem. As a further contribution, we derive the first and second representation theorems of Kato, and we also provide a short and natural proof of the existence and properties of the Friedrichs extension of a densely defined positive operator. It is also worth emphasizing that all results and proofs presented here are valid over both real and complex Hilbert spaces.
\section{Representation of positive symmetric forms}

In what follows, $\hil$ denotes a real or complex Hilbert space and $\tform$ is a non-negative symmetric form on $\hil$, defined on a dense linear subspace $\D$ of $\hil$. In other words, $\tform$ is a semi-inner product. (We emphasize that we do not assume $\tform$ to be  closed or closable.) 

At this point, we note that it is customary to associate a new Hilbert space to the form $\tform$, namely by completing the pre-Hilbert space $(\D, \sip{\cdot}{\cdot}_1)$, where the inner product is defined by
\[
\sip{\cdot}{\cdot}_1 := \tform(\cdot,\cdot) + \sip{\cdot}{\cdot}.
\]
It is easy to see that the closedness of the form $\tform$ is equivalent to the completeness of the pre-Hilbert space $(\D, (\cdot,\cdot)_1)$ (see \cite{kato}). However, instead of this construction, we will use the following more natural approach, which is better suited to our purposes.
Let 
\begin{equation*}
    \ker \tform\coloneqq \set{x\in \D}{\tform(x,x)=0},
\end{equation*}
called the kernel of $\tform$, which is a linear subspace of $\D$, due to the Cauchy-Schwarz inequality. Then the factor space $\D/\ker \tform$ becomes a pre-Hilbert space with the inner product
\begin{equation*}
    \sipt{h+\ker \tform}{g+\ker \tform},\qquad h,g\in \D.
\end{equation*}
The completion of $\D/\ker\tform$ with the above inner product will be denoted by $\hilt$. The canonical map $Q$ of $\D\subseteq \hil$ to $\hilt$ is given by 
\begin{equation}\label{E:Q}
    Qg\coloneqq g+\ker\tform\qquad(g\in\D).
\end{equation}
We shall use the obvious fact that the range of $Q$ forms a dense linear subspace in $\hilt$. 

Of course, at this point we cannot say anything about regularity properties of $Q$, such as closedness or continuity. However, an elementary representation theorem already holds for $\tform$ as follows:
\begin{lemma}\label{L:L1}
    Let $\tform $ be any nonnegative symmetric form on the dense subspace $\D$ of the real or complex Hilbert space $\hil$. Then 
     \begin{equation}\label{E:QQgh}
         \tform(g,h)=\sip{Q^*Qg}{h},\qquad g\in\DD, h\in\D,
     \end{equation}
     where 
     \begin{equation}\label{E:DD}
         \DD\coloneqq\dom Q^*Q=\set[\Big]{g\in\D}{\sup\set{\abs{\tform(g,h)}^2}{h\in\D, \sip hh\leq 1}<+\infty}.
     \end{equation}
\end{lemma}
\begin{proof}
    It is clear that 
    \begin{equation*}
        \tform(g,h)=\sipt{Qg}{Qh}=\sip{Q^*Qg}{h},
    \end{equation*}
    for $g\in\dom Q^*Q$ and $h\in\D$. Furthermore, we also have
    \begin{align*}
        \DD=\dom Q^*Q&=\set{g\in\D}{Qg\in\dom Q^*}\\
        &=\set[\Big]{g\in\D}{\sup\set{\abs{\sipt{Qg}{Qh}}^2}{h\in\D, \sip hh\leq 1}<+\infty}\\
        &=\set[\Big]{g\in\D}{\sup\set{\abs{\tform(g,h)}^2}{h\in\D, \sip hh\leq 1}<+\infty},
    \end{align*}
    which proves \eqref{E:DD}.
\end{proof}
\begin{remark}
 It is clear by the very definition of $\DD$ that the following two statements are equivalent to each other:
 \begin{enumerate}[label=\textup{(\roman*)}]
     \item $g\in \DD$,
     \item $g\in \D$ and there exists a constant $C_g \geq 0$ such that 
     \begin{equation}\label{E:DD2}
         \abs{\tform(g,h)}^2\leq C_g\cdot \sip{h}{h}\qquad \mbox{for all $h\in \D$}.
     \end{equation}
 \end{enumerate}
\end{remark}
In what follows, we focus on the representability of the form $\tform$ by positive self-adjoint operators. By a representation we mean a positive and self-adjoint operator $S$ on $\hil$ which satisfies 
\begin{equation}\label{E:reprezdef}
    \sip{Sg}{h}=\tform(g,h)\qquad (g\in\D\cap\dom S, h\in \D).
\end{equation}
Of course, the set $\D \cap \dom S$ can be rather "thin" (even the trivial subspace $\{0\}$). The interesting case is when it is "large enough" to provide nontrivial information about the form $\tform$.

 By the following simple observation, this intersection is always a subset of $\DD$. In other words, $\DD$ is the largest possible set on which the form $\tform$ can be represented by an operator.
\begin{proposition}\label{P:2.2}
    Let $\tform$ be a positive symmetric form in the real or complex Hilbert space $\hil$. If $S$ is a positive self-adjoint operator representation of $\tform$ in the sense of \eqref{E:reprezdef}, then
    \begin{equation*}
        \dom S\cap \D\subseteq \DD.
    \end{equation*}
\end{proposition}
\begin{proof}
If $g\in \dom S\cap \D$ is any vector, then 
\begin{equation*}
    \abs{\tform(g,h)}^2=\abs{\sip{Sg}{h}}^2\leq \|Sg\|^2\|h\|^2
\end{equation*}
for every $h\in\D$, hence $g\in\DD$.
\end{proof}
In the following theorem, we show that any densely defined positive form (without any additional regularity assumptions) admits a representation in a maximal sense; that is, there exists a representing operator $T$ such that $\dom T \cap \D = \DD$. This theorem is the main result of the section.
\begin{theorem}\label{T:represent1}
 Under the assumptions of Lemma \ref{L:L1}, there exists a positive self-adjoint operator $T$ in $\hil$ with the following properties: 
 \begin{enumerate}[label=\textup{(\roman*)}]
     \item $\DD= \dom T\cap \D$,
     \item $T$ is a representing operator of $T$ in the sense that  
 \begin{equation}\label{E:reprez1}
     \tform(g,h)=\sip{Tg}{h} \qquad(g\in \DD,h\in \D),
 \end{equation}
 \item $\D\subseteq \dom T^{1/2}$ and
 \begin{equation}\label{E:T23iii}
     \tform(g,h)=\sip{T^{1/2}g}{T^{1/2}h},\qquad (g\in \DDD,h\in \D),
 \end{equation}
  where 
  \begin{equation}\label{D:DDD}
      \DDD=\set{g\in \D}{\exists (g_n)\subset \D_* \mbox{~s.t.~} g_n\to g \mbox{~in $\hil$ and } \tform(g_n-g,g_n-g)\to0}
  \end{equation}
 \end{enumerate}
\end{theorem}
\begin{proof}
    Consider the vector space 
    \begin{equation*}
        Q^*(\DD+\ker \tform)\coloneqq\set{Q^*(g+\ker\tform)}{g\in \DD}
    \end{equation*}
    equipped with the following inner product 
    \begin{equation*}
        \sipQ{Q^*(g+\ker\tform)}{Q^*(h+\ker\tform)}\coloneqq \tform(g,h),\qquad g,f\in \DD.
    \end{equation*}
    Note that $\sipQ\cdot\cdot$ is indeed a well defined inner product on $Q^*(\DD+\ker \tform)$ since $\tform(g,g)=0$ implies $Q^*(g+\ker\tform)=0$. Let us denote by $\hil_*$ the Hilbert completion of that pre-Hilbert space so obtained, so that $Q^*(\DD+\ker \tform)\subseteq \hil_{*}$ forms a dense linear subspace in with respect to the norm induced by $\sipQ\cdot\cdot$. 

    Consider now the canonical  embedding operator $J$ of $\hil_*\supseteq Q^*(\DD+\ker \tform)$ into $\hil$, defined by the identification
    \begin{equation}
        JQ^*(g+\ker \tform)\coloneqq Q^*(g+\ker \tform)\qquad (g\in \DD).
    \end{equation}
    We are going to show first that $J:\hil_*\to\hil$ is closable, namely, the domain of its adjoint $J^*$ contains the dense subspace $\D\subseteq \hil$:
    \begin{equation}\label{E:DsubdomJ*}
        \D\subseteq \dom J^*.
    \end{equation}
    Indeed, consider a vector $h\in\D$, then for every $g\in\DD$ we have 
    \begin{align*}
        \abs{\sip{ JQ^*(g+\ker \tform)}{ h}}^2&=\abs{\sip{ Q^*Qg}{ h}}^2=\abs{\sipt{ Qg}{ Qh}}^2\\&=\abs{\tform(g,h)}^2\leq \tform(h,h)\tform(g,g)\\
        &=\tform(h,h)\cdot \sipQ{ Q^*(g+\ker \tform)}{Q^*(g+\ker \tform)}, 
    \end{align*}
    according to the Cauchy-Schwarz inequality. Hence $h\in\dom J^*$, as stated. 

    Next we claim that 
    \begin{equation}\label{E:J*g}
        J^*g= Q^*(g+\ker \tform),\qquad (g\in \DD).
    \end{equation}
    To this aim, consider a vector $g\in\DD$, then
    \begin{align*}
        \sipQ{J^*g}{Q^*(h+\ker \tform)}&=\sip{g}{JQ^*(h+\ker \tform)}=\sip{g}{Q^*(h+\ker \tform)}\\
        &=\sipt{Qg}{h+\ker \tform}=\sipt{h+\ker \tform}{h+\ker \tform}\\
        &=\tform(g,h)=\sipQ{Q^*(g+\ker \tform)}{Q^*(h+\ker \tform)}
    \end{align*}
    for every $h\in\DD$. Thus $J^*g-Q^*(g+\ker \tform)$ is orthogonal to the dense set $Q^*(\DD+\ker\tform)$ in $\hil_*$ which in turn implies \eqref{E:J*g}. 

   Finally, we show that the positive self-adjoint $T\coloneqq J^{**}J^*$ has the properties listed in the theorem. On the first hand, by \eqref{E:J*g}, for every $g\in\DD$ we have 
   \begin{equation*}
   J^*g= Q^*(g+\ker \tform)\in \dom J\subseteq \dom J^{**},    
   \end{equation*}
   hence $g\in \dom J^{**}J^*$. On the  other hand,
\begin{equation*}
    J^{**}J^{*}g=J^{**}Q^*(g+\ker \tform)=Q^*(g+\ker \tform),
\end{equation*}
thus 
\begin{equation*}
    \sip{J^{**}J^{*}g}{h}=\sip{Q^*(g+\ker \tform)}{h}=\sipt{g+\ker \tform}{h+\ker \tform}=\tform(g,h),
\end{equation*}
for every $g\in \DD$ and $h\in \D$, that proves \eqref{E:reprez1}. Consequently, $\dom J^{**}J^*\cap \D\subseteq \DD$ by Proposition \ref{P:2.2}, thus $\dom J^{**}J^*\cap \D\subseteq \DD$.

Finally, to prove (iii) observe first that
\begin{equation*}
    \D\subseteq \dom J^*=\dom (J^{**}J^*)^{1/2}
\end{equation*}
according to \eqref{E:DsubdomJ*}. Let now $g\in\DDD$ and choose a sequence $(g_n)$ in $\DD$ such that $g_n\to g$ in $\hil$ and that $\tform(g_n-g,g_n-g)\to0$. Then clearly $\tform(g_n-g_m,g_n-g_m)\to0$ as $n,m\to\infty$, hence also
\begin{equation*}
    \|T^{1/2}(g_n-g_m)\|^2=\sip{T(g_n-g_m)}{g_n-g_m}=\tform(g_n-g_m,g_n-g_m)\to0,
\end{equation*}
which yields $T^{1/2}g_n\to T^{1/2}g$ by closedness. Consequently,
\begin{equation*}
    t(g_n,h)=\sip{Tg_n}{h}=\sip{T^{1/2}g_n}{T^{1/2}h}\to\sip{T^{1/2}g}{T^{1/2}h}
\end{equation*}
and
\begin{equation*}
    \abs{\tform(g_n,h)-\tform(g,h)}^2=\abs{\tform(g_n-g,h)}^2\leq \tform(g_n-g,g_n-g)\tform(h,h)\to0
\end{equation*}
on the other hand, which together yield \eqref{E:T23iii}.
\end{proof}
The operator $T = J^{**}J^*$ constructed in the above theorem is extremal in the sense that it represents the form $\tform$ on the largest possible set. In what follows, we also show that it is extremal with respect to the natural ordering among all representations of $\tform$. 
To this end, we provide an explicit formula for the quadratic form associated with $T$, as follows:

\begin{lemma}\label{L:L2}
    Let $\tform$ be a positive symmetric form on $\D$ as in the preceding theorem and let $T$ denote the representing operator of $\tform$ constructed in the proof of Theorem \ref{T:represent1}. Then 
\begin{equation}\label{E:domT}
    \dom T^{1/2}=\set[\Big]{h\in \hil}{\sup\set{\abs{\sip{h}{Q^*Qg}}^2}{g\in \DD, \tform(g,g)\leq1}<+\infty},
\end{equation}
    and
    \begin{equation}\label{E:T1/2h}
        \|T^{1/2}h\|^2 = \sup\set{\abs{\sip{h}{Q^*Qg}}^2}{g\in \DD, \tform(g,g)\leq1}\qquad (h\in \dom T^{1/2}).
    \end{equation}
\end{lemma}
\begin{proof}
    Recall from the proof of Theorem \ref{T:represent1} that we have $T=J^{**}J^*$, whence
    \begin{align*}
        \dom T^{1/2}&=\dom J^*\\&
        =\set[\Big]{h\in \hil}{\sup\set{\abs{\sip{h}{JQ^*(g+\ker\tform )}}^2}{g\in \DD, \sip{Q^*Qg}{g}\leq1}<+\infty}\\
        &=\set[\Big]{h\in \hil}{\sup\set{\abs{\sip{h}{Q^*Qg}}^2}{g\in \DD, \tform(g,g)\leq1}<+\infty}.
    \end{align*} 
    Using the density of $ Q^*(\DD+\ker \tform)$ in $\hil_*$, we obtain that for every $h\in\dom T^{1/2}$ 
    \begin{align*}
        \|T^{1/2}h\|^2=\|J^*h\|^2_*&=\sup\set{\abs{\sipQ{J^*h}{Q^*(g+\ker\tform )}}^2}{g\in \DD, \sip{Q^*Qg}{g}\leq1}\\
        &=\sup\set{\abs{\sip{h}{Q^*Qg}}^2}{g\in \DD, \tform(g,g)\leq1},
    \end{align*}
    which proves the statements of the lemma.
\end{proof}
Before stating our next theorem, let us recall the form order `$\preceq$' on the set of positive self-adjoint operators: we write $T \preceq S$ if $\dom S^{1/2} \subseteq \dom T^{1/2}$ and
\begin{equation*}
    \|T^{1/2}h\|^2 \leq \|S^{1/2}h\|^2 \qquad \text{for all } h \in \dom S^{1/2}.
\end{equation*}
For properties of the order $\preceq$, see e.g.\ \cite{Schmudgen}.
\begin{theorem}
    Let $\tform$ be a densely defined positive symmetric form in $\hil$. Let $S$ be any representing operator of $\tform$ in the sense of \eqref{E:reprezdef} such that 
    \begin{equation}\label{E:ScapD}
        \DD\subseteq \dom S\cap \D.
    \end{equation}
    Then $\dom S^{1/2}\subseteq \dom (J^{**}J^*)^{1/2}$ and 
    \begin{equation*}
        \|S^{1/2}h\|^2\geq \|(J^{**}J^*)^{1/2}h\|^2\qquad (h\in \dom S^{1/2}).
    \end{equation*}
    In other words, $J^{**}J^*\preceq S$ holds for every positive self-adjoint representation $S$ of the form $\tform$.
\end{theorem}
\begin{proof} 
Let $S$ be a positive self-adjoint operator satisfying \eqref{E:reprezdef} and \eqref{E:ScapD}. First observe that 
\begin{equation*}
    \sip{Sg}{h}=\tform(g,h)=\sip{Q^*Qg}{h},\qquad g\in\DD, h\in \D,
\end{equation*}
which implies $Sg=Q^*Qg$ for every $g\in\D_*$. 

Let now $h\in\dom S^{1/2}$, then one has 
\begin{align*}
    \abs{\sip{h}{Q^*Qg}}^2&=\abs{\sip{h}{Sg}}^2=\abs{\sip{S^{1/2}h}{S^{1/2}g}}^2\\ 
    & \leq \|S^{1/2}h\|^2\|S^{1/2}g\|^2=\|S^{1/2}h\|^2\cdot \tform(g,g)
\end{align*}
for every $g\in\DD$. Hence $h\in \dom T^{1/2}$ according to Lemma \ref{L:L2}. Note on the other hand, that $\dom S$ is a core for  $S^{1/2}$, hence for every $h\in \dom S^{1/2}$ one obtains that 
\begin{align*}
    \|S^{1/2}h\|^2&=\sup\set{\abs{\sip{S^{1/2}h}{S^{1/2}f}}^2}{f\in\dom S, \sip{Sf}{f}\leq 1}\\
    &\geq \sup\set{\abs{\sip{h}{Sg}}^2}{g\in\DD, \sip{Sg}{g}\leq 1}\\
    &=\sup\set{\abs{\sip{h}{Q^*Qg}}^2}{g\in\DD, \tform(g,g)\leq 1}\\
    &=\|T^{1/2}h\|^2,
\end{align*}
which proves the statement of the theorem.
\end{proof}
\begin{corollary}
    Let $\tform$ be a positive symmetric form in the real or complex Hilbert space $\hil$ with dense domain $\D$.
    Assume that for every $h\in\D$ there exists a sequence $\seq g$ in $\DD$ such that $g_n\to h$ and $\tform(h-g_n,h-g_n)\to 0$, that is, 
     \begin{equation}\label{E:DDD=D}
    \DDD=\D.     
     \end{equation}
     Then there exists a positive self-adjoint operator $T$ in $\hil$ that fulfills the following properties:
        \begin{enumerate}[label=\textup{(\alph*)}]
            \item $\DD\subseteq \dom T$ and $\tform(g,h)=\sip{Tg}{h}$ for $g\in\D\cap\dom T$ and $h\in\D$,
            \item $\D\subseteq \dom T^{1/2}$ and $\tform(g,h)=\sip{T^{1/2}g}{T^{1/2}h}$ for $g,h\in\D$.
        \end{enumerate}
\end{corollary}
\begin{proof}
    Let $T$ denote the positive self-adjoint operator satisfying properties (i)-(iii) of Theorem \ref{T:represent1}. We show that $T$ possesses the desired properties. Indeed, we have $\D=\DDD\subseteq\dom T^{1/2}$ by \eqref{E:DDD=D} and thus 
    \begin{equation*}
        \tform(g,h)=\sip{T^{1/2}g}{T^{1/2}h}\qquad (h,g\in\D)
    \end{equation*}
    according to \eqref{E:T23iii}. This in particular yields   
    \begin{equation*}
        \tform(g,h)=\sip{T^{1/2}g}{T^{1/2}h}=\sip{Tg}{h}
    \end{equation*}
    for $g\in\dom T\cap \D$ and $h\in\D$.
\end{proof}

Before stating Kato's representation theorems, we recall the definition of a closed form. A positive symmetric form $\tform$ (with domain $\D$) is said to be closed if and only if for every sequence $\seq g$ in $\D$, the conditions $g_n \to g$ and $\tform(g_n - g_m, g_n - g_m) \to 0$ as $n,m \to +\infty$ imply that $g \in \D$ and $\tform(g_n - g, g_n - g) \to 0$.

As an immediate consequence of Theorem~\ref{T:represent1}, we recover Kato's first and second representation theorems for densely defined, closed, positive symmetric forms:
\begin{corollary}
    Let $\tform$ be a densely defined closed form in the real or complex Hilbert space $\hil$. Then there exists a unique positive self-adjoint operator $T$ in $\hil$ satisfying the following properties 
     \begin{enumerate}[label=\textup{(\alph*)}]
            \item $\dom T\subseteq \D$ and $\tform(g,h)=\sip{Tg}{h}$ for $g\in\dom T$ and $h\in\D$,
            \item $\D=\dom T^{1/2}$ and $\tform(g,h)=\sip{T^{1/2}g}{T^{1/2}h}$ for $g,h\in\D$.
        \end{enumerate}
\end{corollary}
\begin{proof}
    Note that the closedness of the form $\tform$ implies that the canonical operator $Q:\hil\to\hilt$ is closed. Consequently, the positive operator $T=Q^*Q$ is self-adjoint  with domain $\dom T=\DD\subseteq \D$. Hence property (a) is immediate from \eqref{E:QQgh}. Furthermore we have that  $\dom (Q^*Q)^{1/2}=\dom Q=\D$, hence
    \begin{equation*}
        \sip{T^{1/2}g}{T^{1/2}h}=\sipt{Qg}{Qh}=\tform(g,h),\qquad g,h\in\D,
    \end{equation*}
    which proves statement (b).
\end{proof}
%\section{Representation of closable forms}
Recall that a positive symmetric form $\tform$ is said to be \emph{closable} if, for every sequence $\seq g$ in $\D$, conditions $g_n \to 0$ and $\tform(g_n - g_m, g_n - g_m) \to 0$ as $n, m \to +\infty$ imply that $\tform(g_n, g_n) \to 0$.

For closable forms, the following representation theorem holds:
\begin{corollary}
    Let $\tform$ be a nonnegative symmetric closable form in the dense subspace $\D\subseteq \hil$. Then the canonical embedding $Q:\hil\supset \D\to\hilt$ is closable and $Q^{*}Q^{**}$ is  a representation of $\tform$: 
    \begin{enumerate}[label=\textup{(\alph*)}]
        \item $\tform(g,h)=\sip{Q^*Q^{**}g}{h}$ for $g\in\DD$ and $h\in \D$,
        \item $\tform(g,h)=\sip{(Q^*Q^{**})^{1/2}g}{(Q^*Q^{**})^{1/2}h}$ for $g,h\in\D$.
    \end{enumerate}
\end{corollary}
\begin{proof}
    Clearly, $\tform$ is closable if and only if the canonical embedding $Q$ in \eqref{E:Q} is a closable operator,  with closure $Q^{**}$. Then $Q^*Q^{**}$ is a positive self-adjoint operator in $\hil$ such that 
    \begin{equation*}
        \D\subseteq \dom Q^{**}=\dom (Q^{**}Q^*)^{1/2}
    \end{equation*}
    It is then clear that $Q^*Q^{**}$ satisfies (a) and (b).
\end{proof}

Note that if $\tform$  a closable, then the form 
\begin{equation*}
    \bar{\tform}(h,k)\coloneqq \sip{(Q^*Q^{**})^{1/2}g}{(Q^*Q^{**})^{1/2}h}\qquad (h,k\in\dom Q^{**})
\end{equation*}
is equal to the closure of $\tform$.

\begin{theorem}\label{T:2.7} Let $\tform$ be a densely defined non-negative form in the real or complex Hilbert space $\hil$, with a dense domain $\D$. Let $T$ be a positive self-adjoint representation of $\tform$  with domain $\dom T=\DD$, such that
\begin{equation*}
    \tform(g,h)=\sip{Tg}{h},\qquad (g\in\dom T,h\in \D).
\end{equation*}
Then the following statements are equivalent:
\begin{enumerate}[label=\textup{(\roman*)}]
    \item $\dom T^{1/2}=\D$ and $\tform(g,h)=\sip{T^{1/2}g}{T^{1/2}h}$ for $g,h\in\D$,
    \item  the following two conditions are fulfilled: 
    \begin{enumerate}[label=\textup{(\alph*)}]
            \item for each $g\in\D$ there exists a  sequence $\seq g\subset \DD$, such that $g_n\to g$ and $\tform(g_n-g,g_n-g)\to0$,
            \item for each sequence $\seq g\subset \DD$, such that $g_n\to g$ and $\tform(g_n-g_m,g_n-g_m)\to0$, it follows that $g\in\D$.
        \end{enumerate}
\end{enumerate}
\end{theorem}
\begin{proof}
    First, we prove that (i) implies (ii). To this end, consider a vector $g\in\D=\dom T^{1/2}$. Since $\dom T$ is core for $T^{1/2}$,  there exists $\seq g\subset \dom T=\DD $ such that $g_n\to g$ and $T^{1/2}g_n\to T^{1/2}g$. Therefore,
    \begin{equation*}
        \tform(g_n-g,g_n-g)=\|T^{1/2}(g_n-g)\|^2\to0,
    \end{equation*}
    which proves (a). Consider now a sequence $\seq g$ of $\DD$ such that $g_n\to g$ and $\tform(g_n-g,g_n-g)\to0$ for some $g\in\hil$. Then 
    \begin{equation*}
        \|T^{1/2}(g_n-g_m)\|^2=\tfrom(g_n-g_m,g_n-g_m)\to0,
    \end{equation*}
    hence $g\in\dom T^{1/2}=\D$, thanks to the closedness of $T^{1/2}$. 

    We now prove that (ii) implies (i). Consider first $g\in\D$ and choose $\seq g\subset \DD=\dom T$ as in (a). Then 
    \begin{equation*}
        \|T^{1/2}(g_n-g_m)\|^2=\sip{T(g_n-g_m)}{g_n-g_m}=\tform(g_n-g_m,g_n-g_m)\to0,
    \end{equation*}
    which in turn implies $g\in\dom T^{1/2}$ and $T^{1/2}g_n\to T^{1/2}g$. In particular, we obtain that $\D\subset \dom T^{1/2}$ and that 
    \begin{equation*}
        \sip{T^{1/2}g}{T^{1/2}h}=\tform(g,h),\qquad g,h\in \D. 
    \end{equation*}
    To prove the inclusion in the other direction, let $g \in \dom T^{1/2}$, and let $\seq g$ be a sequence in $\dom T = \D$ such that $g_n \to g$ and $T^{1/2}g_n \to T^{1/2}g$. Then, using the fact that
    \begin{equation*}
            \tform(g_n - g_m, g_n - g_m) = \sip{T(g_n - g_m)}{g_n - g_m}=\|T^{1/2}(g_n - g_m)\|^2 \to 0,
    \end{equation*}
    and taking into account point (b), we obtain $g \in \mathcal{D}$. Consequently, the inclusion $\dom T^{1/2} \subseteq \mathcal{D}$ also holds.
\end{proof}
\begin{corollary}
    Let $T$ be a densely defined (not necessarily closed) operator in a real or complex Hilbert space $\hil$ such that $T^*T$ is a self-adjoint operator. Then the following two statements are equivalent:
    \begin{enumerate}[label=\textup{(\roman*)}]
        \item $\dom T=\dom (T^*T)^{1/2}$ and $\sip{Tg}{Th}=\sip{(T^*T)^{1/2}h}{(T^*T)^{1/2}g}$ for every $g,h\in \dom T$,
        \item the following two conditions are fulfilled: 
    \begin{enumerate}[label=\textup{(\alph*)}]
            \item  for every $g\in\dom T$ there is a sequence $\seq g\subset \dom (T^*T)$, such that $g_n\to g$ and $Tg_n\to Tg$,
            \item for each sequence $\seq g\subset \dom(T^*T)$, such that $g_n\to g$ and $T(g_n-g_m)\to0$, it follows that $g\in\dom T$.
        \end{enumerate}     
    \end{enumerate}
\end{corollary}
\begin{proof}
    Consider the form $\tform$ by
    \begin{equation*}
        \tform(g,h)\coloneqq \sip{Tg}{Th},\qquad g,h\in\dom T.
    \end{equation*}
    Then, clearly, $\dom T^*T=\DD$ and
    \begin{equation*}
        \tform(g,h)=\sip{T^*Tg}{h},\qquad g\in\dom T^*T, h\in\dom T.
    \end{equation*}
    Hence the statement follows from Theorem \ref{T:2.7}.    
\end{proof}

\begin{remark}
It is well known that if $T$ is a densely defined closed operator, then both $T^*T$ and $TT^*$ are self-adjoint, see \cites{Neumann,Schmudgen,Stone}. However, it is not difficult to construct an example where $T^*T$ (or $TT^*$) is self-adjoint even though $T$ itself is not closed. (For example, let $S$ be a closed operator and let $T$ be its restriction to $\dom T^*T$; then $T^*T$ is self-adjoint, but $T$ is typically not closed.) 

However, if both $T^*T$ and $TT^*$ are self-adjoint, then $T$ must necessarily be closed; see \cites{SZ-TZS:reversed, SZTZS-Adjoint}.
\end{remark}
%%%%%%%%%%%%%%%%%%%%%%%%%%%%%%%%%%%%%%%%%%%%%%%%%%%%%%%%%%%%%%%%%%%%%%%%%%%
\section{The Friedrichs extension of positive operators}
Using the construction provided in the previous section, we now present the so-called Friedrichs extension of a densely defined positive operator $T$. This distinguished extension is the maximal element with respect to the form order `$\preceq$' among all positive self-adjoint extensions of $T;$
cf.\ \cites{AndoNishio,Friedrichs,Freudenthal}; see also \cites{Prokaj0,Prokaj,SebTarcsayWasaa,SebTarcsayOpuscula}. Although the result itself is classical, the construction we provide here is as natural and perhaps as simple as possible.

\begin{theorem}\label{T:Friedrichs}
    Let $T$ be a positive symmetric operator in the real or complex Hilbert space $\hil$, with a dense domain $\D\coloneqq \dom T$. Then there exists a positive self-adjoint operator $T_F$ of $T$ such that
    \begin{enumerate}[label=\textup{(\alph*)}]
        \item $T\subset T_F$, that  is, $T_F$ is a positive self-adjoint extension of $T$,
        \item $S\preceq T_F$ for every positive self-adjoint extension $S$ of $T$.
    \end{enumerate}
\end{theorem}
\begin{proof}
    Let us introduce the nonnegative form \begin{equation*}
        \tform(g,h)\coloneqq \sip{Tg}{h},\qquad g,h\in\D,
    \end{equation*}
    and consider the corresponding Hilbert space $\hilt$ and embedding operator $Q:\hil\supseteq \D\to\hilt$ of \eqref{E:Q}. We claim that 
    \begin{equation*}
        \D=\DD.
    \end{equation*}
    For let $g\in\D$, then
    \begin{align*}
        \abs{\tform(h,g)}^2=&\abs{\sip{Th}{g}}^2=\abs{\sip{h}{Tg}}^2\leq \|h\|^2\|Tg\|^2,
    \end{align*}
    whence $g\in\DD$ according to \eqref{E:DD}.
As a direct consequence, we obtain that the adjoint operator $Q^*$ is densely defined, since it contains the dense subspace $ \ran Q = \DD / \ker \mathfrak{t} \subseteq \hilt$. Consequently, $ Q $ is closable. On the other hand, for all vectors $ h, g \in \D $, we have
\[
\sip{Tg}{h}  = \mathfrak{t}(g, h) = \sip{Q^*Q g}{h},
\]
therefore $T = Q^* Q$. This implies that $T_F \coloneqq Q^* Q^{**}$ is a positive self-adjoint operator extending $ T $.

Finally, we show that $ T_F $ is maximal among the positive self-adjoint extensions of $ T $. Let $ T \subset S $ be an arbitrary positive self-adjoint extension of $T$. Consider the form associated with $ S $, given by
\[
\mathfrak{s}(g, h) \coloneqq \sip{ Sg}{h} \qquad (g, h \in \dom S),
\]
and let $ \hil_{\mathfrak{s}} $ be the corresponding auxiliary Hilbert space, with embedding operator $ Q_{\mathfrak{s}} $. Then, taking into account that $ T \subset S $, we have $ Q \subset Q_{\mathfrak{s}} $, and therefore $ Q^{**} \subset Q_{\mathfrak{s}}^{**} $. Moreover, according to the first part of the proof, $ S = Q_{\mathfrak{s}}^{*} Q_{\mathfrak{s}}^{**} $, since $ S $ does not admit any proper self-adjoint extension. As a consequence,
\begin{equation*}
    \dom T_F^{1/2}=\dom Q^{**}\subseteq \dom Q^{**}_{\sform}=\dom S^{1/2},
\end{equation*}
and 
\begin{equation*}
    \|T_F^{1/2}h\|^2=\|Q^{**}h\|^2=\| Q^{**}_{\sform}h\|^2=\| S^{1/2}h\|^2
\end{equation*}
for every $h\in\dom T_F^{1/2}$. Hence $S\preceq T_F$, as stated. 
\end{proof}

\begin{theorem}
    Let $\tform$ be of nonnegative symmetric form such that 
    \begin{equation*}
        \DD=\D.
    \end{equation*}
    Then $\tform $ is closable and the Friedrichs extension of the positive operator $T\coloneqq Q^*Q$ is equal to $Q^*Q^{**}$, that is, $T_F=Q^*Q^{**}$.
\end{theorem}
\begin{proof}
First, we observe that $\tform$ is closable. Indeed, since $\D = \dom Q$ and és $\DD = \dom Q^*Q$, the assumption $\D = \DD$ implies that $\dom Q^*$ contains the dense subspace $\ran Q=\hilt/\ker \tform$. Consequently, the operator $Q$ is closable, which in turn implies that the form $\tform$ is also closable.

It remains to show that $Q^*Q^{**}$ coincides with the Friedrichs extension of the positive operator $T := Q^*Q$, i.e.,
\begin{equation}
Q^*Q^{**} = T_F.
\end{equation}
To this end, it clearly suffices to show that for any positive self-adjoint extension $S$ of $T$, the relation $S \preceq Q^*Q^{**}$ holds. 

Let $g \in \operatorname{dom} Q^{**}$. Then there exists a sequence $\seq g \subset \D$ such that $g_n \to g$ and $Qg_n \to Q^{**}g$. In this case, we have
 \begin{align*}
     \|S^{1/2}(g_n-g_m)\|^2&=\sip{S(g_n-g_m)}{g_n-g_m}\\&=\sip{T(g_n-g_m)}{g_n-g_m}\\&=\|Q(g_n-g_m)\|^2\to0,\qquad (n,m\to\infty),
 \end{align*}
hence $g\in\dom S^{1/2}$ due to closedness. Also
\begin{align*}
\|(Q^*Q^{**})^{1/2}g\|^2 &= \|Q^{**}g\|^2 = \lim_{n \to \infty} \|Qg_n\|^2 \\&= \lim_{n \to \infty} \sip{Tg_n}{g_n} = \lim_{n \to \infty} \|S^{1/2}g_n\|^2=\|S^{1/2}g\|^2.    
\end{align*}
Putting all of the above together, we conclude that $S \preceq Q^*Q^{**}$, and hence $Q^*Q^{**}$ indeed coincides with the Friedrichs extension of $T$.
\end{proof}
According to the above theorem, the equality $\D = \DD$ implies that the form $\tform$ is closable. However, the following statement shows that the closedness of $\tform$ under the same assumption holds only under a very specific condition -- namely, when $\tform$ is a bounded form.

\begin{proposition}\label{P:3.3}
    Let $\tform$ be a closed form in the real or complex Hilbert space $\hil$ with dense domain $\D$. Then the following assertions are equivalent:
    \begin{enumerate}[label=\textup{(\roman*)}]
        \item $\D=\DD$,
        \item $\tform$ is bounded.
    \end{enumerate}
\end{proposition}
\begin{proof}
Clearly, if the form $\tform$ is closed and continuous, then $\D = \DD = \hil$.  
Conversely, assume that condition (i) holds for a closed form $\tform$. In this case, $Q$ is a closed operator for which
\begin{equation*}
\dom Q = \dom Q^*Q.
\end{equation*}
Consider the “flip” operator $V : \hilt  \times \hil \to \hil \times \hilt$ defined by $V\{h, g\} := \{-g, h\}$. Then the following well-known relation holds between the graphs of the operators $Q$ and $Q^*$:
\begin{equation*}
    \hilt \times \hil = G(Q^*) \oplus V[G(Q)].
\end{equation*}
In particular, any vector $z \in \hilt$ can be written in the form $z = u - Qg$ for suitable $g \in \dom Q$ and $u \in \dom Q^*$.  
By assumption (i), we have $Qg \in \dom Q^*$, hence $z \in \dom Q^*$, which implies that $\dom Q^* = \hilt$. By the Closed Graph Theorem, it follows that $Q^*$ is a bounded operator, and therefore the form $\tform$ is also bounded.
\end{proof}

The following theorem addresses the case when the form $\tform$ satisfies $\D = \DD$. According to Proposition \ref{P:3.3}, in this situation the form $\tform$ is either bounded or not closed.

\begin{theorem}
    Let $\tform$ by a densely defined nonnegative form on a real or complex Hilbert space $\hil$ satisfying $\DD=\D$, that  is, 
    \begin{equation}
        \sup\set{\abs{\tform(g,h)}^2}{h\in\D,\sip hh\leq 1}\qquad \mbox{for all $g\in\D$}. 
    \end{equation}
    Then $\tform$ is closable, and the Friedrichs extension of $Q^*Q$ is identical to $Q^*Q^{**}$: 
    \begin{equation}
        Q^*Q^{**}=(Q^*Q)^{}_F,
    \end{equation}
    where $Q:\hil\to\hilt$ is the canonical quotient operator in \eqref{E:Q}. 
\end{theorem}
\begin{proof}
    By assumption, $\dom Q^*Q = \dom Q$, which implies that $\ran Q \subseteq \dom Q^*$. This, in turn, means that $Q^*$ is densely defined, and thus $Q$ is closable with $Q^{**} = \overline{Q}$. Consequently, the form $\mathfrak{t}$ is closable, and its closure is given by
\begin{equation*}
    \overline{\mathfrak{t}}(g, h) = \sip{Q^{**} g}{Q^{**} h}, \qquad g,h\in\dom Q^{**}.
\end{equation*}
    Clearly, $S$ is a positive self-adjoint extension of $Q^{*}Q^{**}$, we show that $S\preceq Q^*Q^{**}$. For let $g\in\dom Q^{**}$ and consider a sequence $\seq g$ of $\D$ such that $g_n\to g$ and $Qg_n\to Q^{**}g$. Then 
    \begin{align*}
        \|S^{1/2}(g_n-g_m)\|^2&=\sip{S(g_n-g_m)}{g_n-g_m}=\sip{Q^*Q(g_n-g_m)}{g_n-g_m}\\&=\|Q(g_n-g_m)\|^2\to 0,
    \end{align*}
    thus $g\in \dom S^{1/2}$ and $S^{1/2}g_n\to S^{1/2}g$. Hence 
    \begin{equation*}
        \|S^{1/2}g\|^2=\limn \|S^{1/2}g_n\|^2=\limn\|Qg_n\|^2= \|Q^{**}g\|^2=\|(Q^*Q^{**})^{1/2}g\|^2.
    \end{equation*}
    Therefore, $S\preceq Q^*Q^{**}$, which means that $Q^*Q^{**}$ is identical with the Friedrichs extension of $Q^*Q$.
\end{proof}

%\bibliographystyle{abbrvnat}
%\bibliography{references}
\end{document}